\documentclass[11pt,reqno]{amsart}

\usepackage{amsmath, amsfonts, amsthm, amssymb}

\newtheorem{Theorem}{Theorem}[section]
\newtheorem{Lemma}{Lemma}[section]

\theoremstyle{definition}
\newtheorem{Definition}{Definition}[section]

\theoremstyle{remark}
\newtheorem{Remark}{Remark}[section]

\numberwithin{equation}{section}

\renewcommand{\u}{{\bf u}}

\newcommand{\R}{{\mathbb R}}
\newcommand{\Dv}{{\rm div}}

\newcommand{\m}{{\bf m}}

\def\f{\frac}
\renewcommand{\O}{\Omega}

\def\D{\Delta }

\def\hf1{^\f{1}{1-\xi^2}}

\def\be{\begin{equation}}
\def\en{\end{equation}}
\def\bs{\begin{split}}
\def\es{\end{split}}

\author{Cheng Yu}
\address{Department of Mathematics,  The University of Texas, Austin, Texas 78712.}
\email{yucheng@math.utexas.edu}

\title%[Global weak solutions to N-S-V equations]
[]
{ Energy conservation for the weak solutions of the compressible Navier-Stokes equations}

\subjclass[2000]{}

\date{\today}

\begin{document}
\begin{abstract} In this paper, we prove the energy conservation for the weak solutions of the compressible Navier-Stokes equations for any time $t>0$, under certain conditions. The results hold for the renormalized solutions of the equations with constant viscosities, as well as the weak solutions of the equations with degenerate viscosity.
Our conditions do not depend on the dimensions.
The energy may conserve on the vacuum for the compressible Navier-Stokes equations with constant viscosities. Our results are the first ones on the energy conservation for the weak solutions of the compressible Navier-Stokes equations.
\end{abstract}

\maketitle
\section{Introduction}
This paper deals with the
energy conservation for the weak solutions of the compressible Navier-Stokes equations, namely
\begin{equation}
\begin{split}
\label{CNS}
&\rho_t+\Dv(\rho\u)=0\\
&(\rho\u)_t+\Dv(\rho\u\otimes\u)+\nabla P-2\mu\D\u-\lambda\nabla\Dv\u=0,
\end{split}
\end{equation}
as well as the following equations with degenerate viscosity
\begin{equation}
\label{NS equation}
\begin{split}
&\rho_t+\Dv(\rho\u)=0\\
&(\rho\u)_t+\Dv(\rho\u\otimes\u)+\nabla P-2\nu\Dv(\rho\mathbb{D}\u)=0,
\end{split}
\end{equation}
respectively with initial data
\begin{equation}
\label{initial data}
\rho|_{t=0}=\rho_0(x),\;\;\;\;\;\rho\u|_{t=0}=\m_0(x)=\rho_0\u_0,
\end{equation}
where $P=\rho^{\gamma},\,\gamma>1,$ denotes the pressure, $\rho$ is the density of fluid, $\u$ stands for the velocity of fluid, $\mathbb{D}\u=\frac{1}{2}[\nabla\u+\nabla^T\u]$ is the strain tensor. The viscosity coefficients $\mu,\,\lambda$ satisfies $\mu>0,\;2\mu+N\lambda\geq 0.$
For the sake of simplicity we will consider the case of bounded domain with periodic boundary conditions in $\R^N$, namely $\O=\mathbb{T}^N$, $N=2, 3$. Here we define $\u_0=0$ on the set $\{x|\rho_0(x)=0\}.$
 Without loss of generality,  we will fix $2\nu=1$ in \eqref{NS equation} from now on.

\vskip0.3cm

As we all known, the global existence of weak solution to \eqref{CNS}, \eqref{initial data} was established in \cite{FNP,F04, L} for any $\gamma>\frac{N}{2}.$
The weak solution of \eqref{NS equation}-\eqref{initial data} was studied in \cite{BD,BD2006,BDL, BDZ,LX, MV, VY}. The purpose of this paper is to provide the sufficient conditions for  the energy conservation of the weak solutions of \eqref{NS equation} first, then extend our result to the weak solutions of \eqref{CNS}.
\vskip0.3cm

A weak solution  $(\rho,\u)$ to \eqref{NS equation}-\eqref{initial data}
 constructed in \cite{VY}
satisfies the following energy inequality \begin{equation}
\label{energy inequality}
\int_{\O}\left(\frac{1}{2}\rho|\u|^2+\frac{\rho^{\gamma}}{\gamma-1}\right)\,dx+\int_0^T\int_{\O}\rho|\mathbb{D}\u|^2\,dx\,dt\leq
\int_{\O}\left(\frac{1}{2}\rho_0|\u_0|^2+\frac{\rho_0^{\gamma}}{\gamma-1}\right)\,dx,
\end{equation}
even if $\rho\geq \underline{\rho}>0.$
A natural question to ask is when a weak solution $(\rho,\u)$ of the compressible Navier-Stokes equations \eqref{NS equation}-\eqref{initial data} satisfies, not only \eqref{energy inequality}, but the stronger version
\begin{equation}
\label{energy equality}
\int_{\O}\left(\frac{1}{2}\rho|\u|^2+\frac{\rho^{\gamma}}{\gamma-1}\right)\,dx+\int_0^T\int_{\O}\rho|\mathbb{D}\u|^2\,dx\,dt=
\int_{\O}\left(\frac{1}{2}\rho_0|\u_0|^2+\frac{\rho_0^{\gamma}}{\gamma-1}\right)\,dx.
\end{equation}
It is
well known that if a solution is smooth enough, then it conserves the energy. Thus, our question is connected to the regularity  of weak solutions.
However, the regularity of global weak solutions of \eqref{NS equation} remains yet mostly open.
Naturally, of particular interest is the question: how badly behaved $(\rho,\u)$ can be so that  it keeps energy conservation? In mathematics, how much regularity is needed for a weak solution to ensure energy equality \eqref{energy equality}? We can ask the same question for the weak solutions of equation \eqref{CNS}.
The main contribution of our
paper is to provide the first result on this topic of the compressible Navier-Stokes equations for any $t>0.$
%We would like to  give the first answer to this interesting question in this short paper.
\vskip0.3cm

As we mentioned before,  the question to ask is how much regularity is necessary to conserve the energy.
 In the context of the incompressible Navier-Stokes equations, the  pioneer study was done by Serrin \cite{Serrin}. He has proved that a weak solution $\u$ conserves its energy globally, provided
 $\u\in L^p(0,T;L^q(\O))$, where
$$\frac{2}{p}+\frac{N}{q}\leq 1,$$
where $N$ is the dimension. Later, Shinbrot \cite{Shinbrot} proved the same conclusion if $$\frac{2}{p}+\frac{2}{q}\leq 1$$
and $q\geq 4.$
Meanwhile, in the context of incompressible Euler equations, this question is linked to
 a famous conjecture of Onsager \cite{O}: energy should be conserved if the solution is Holder continuous with exponent greater than 1/3, while solutions with less regularity
possibly dissipate energy. The first part was solved by \cite{CCFS,CET,Ey} while significant progress has recently been made on the second part \cite{BDIS,BDS}. Very recently,   Isett in \cite{I}  solved the second part of  Onsager conjecture.
  Feireisl-Gwiazda-Swierczewska-Gwiazda-Wiedemann \cite{FGSW} gave sufficient conditions on the regularity of solutions to the inhomogeneous incompressible Euler and the compressible isentropic Euler systems in order for the energy to be conserved in the distribution sense. Leslie-Shvydkoy \cite{LS} showed the energy balance relation for density dependent incompressible Navier-Stokes holds for weak solutions if the velocity, density and pressure belong to a range Besov spaces of smoothness $\frac{1}{3}.$
To our best knowledge, there is no available result for the weak solutions of the compressible Navier-Stokes equations.
% at least for \eqref{NS equation}.

\vskip0.3cm

The sufficient conditions for the energy conservation are addressed in this paper. Our approach relies on the idea of Vasseur-Yu \cite{VY}  and Yu \cite{Yu}.
Compared to the incompressible Navier-Stokes equations, we are not able to deduce  that $\u(t,x)$ is continuous at time $t=0$. Our alternative way for the compressible Navier-Stokes equations is to
gain  the continuity of $\rho(t)$ and $(\sqrt{\rho}\u)(t)$ in the strong topology at $t=0$.  To this end, we need  $\nabla\sqrt{\rho}\in L^{\infty}(0,T;L^2(\O))$. Lucky, BD entropy (see \cite{BD,BD2006,BDL}) gives us such an estimate on density for the degenerate compressible Navier-Stokes equations. However, we need
additional condition on the density to study the weak solutions of \eqref{CNS}. Meanwhile, it is crucial to rely on Lemma \ref{Lions's lemma} to handle the nonlinear compositions $(\rho\u)_t$ and $\rho\u\otimes\u$.
Here we have to mention that our approach allows us to handle vacuum states for the weak solutions of \eqref{CNS}.
%The other is that our conditions do not depend on the dimensions.

\vskip0.3cm

The question addressed here is motivated by the fact that the energy conservation is fundamental both in the physical theory as well as in the mathematical study of
the fluid dynamics. It is natural, therefore, to seek a
rigorous theory which accommodates this question. The results of
this paper effectively achieve this goal by providing a certain condition for the weak solution, for any $t>0$.
To address our main result, we now give a precise definition and discussion of our weak solutions $(\rho,\u)$ to the initial value problem \eqref{NS equation}-\eqref{initial data} in the following sense.
\begin{Definition}
\label{definition of weak soution}
The $(\rho,\u)$ is called a global weak solution to \eqref{NS equation}-\eqref{initial data}, if $(\rho,\u)$ satisfies the following ones, for any $t\in[0,T]$,
\begin{itemize}%\addtolength{\itemsep}{-0.2\baselineskip}
\item \eqref{NS equation} holds in $\mathcal{D'}((0,T)\times\O))$ and the following is satisfied\\
$\rho\geq 0, \quad \rho \in L^{\infty}([0,T];L^{\gamma}(\O)),$
%\\$\rho(1+|\u|^2)\ln (1+|\u|^2)\in L^{\infty}(0,T;L^1(\O)),$
\\$ \nabla\rho^{\frac{\gamma}{2}}\in L^2(0,T;L^2(\O)),\quad\nabla\sqrt{\rho}\in L^{\infty}(0,T;L^2(\O)),$\\
           $  \sqrt{\rho}\u \in L^{\infty}(0,T;L^2(\O)),\quad\sqrt{\rho}\nabla\u\in L^2(0,T;L^2(\O)),$

\item  \eqref{initial data} holds in $\mathcal{D'}(\O)$.
%\item \eqref{energy inequality} holds for almost every $t\in[0,T]$.
 \end{itemize}
\end{Definition}

The global existence of weak solution to \eqref{NS equation} under the above definition was proved in \cite{VY}.
The energy inequality \eqref{energy inequality} holds for almost every $t\in[0,T]$ if the density is bounded away from zero.
%This paper is dedicated to the proof of the following result on the energy conservation.
Here we state the following result on the energy conservation for the weak solution.
\begin{Theorem}
\label{main result} Let $(\rho,\u)$ be a weak solution of \eqref{NS equation}-\eqref{initial data} in the sense of Definition \ref{definition of weak soution}.
Moreover,  if

\begin{equation}
\label{bounded away from zero}
0<\underline{\rho}\leq \rho(t,x)\leq \bar{\rho}<\infty,
\end{equation}

\begin{equation}
\label{condition for velocity}
\u \in L^p(0,T;L^q(\O))\quad\text{ for any }  \frac{1}{p}+\frac{1}{q}\leq \frac{5}{12}\text{ with }\; q\geq 6 ,
\end{equation}

and
\begin{equation}
\label{additional condition on initial data}
\sqrt{\rho_0}\u_0\in L^4(\O),
\end{equation}

 then such a weak solution
 $(\rho,\u)$  satisfies \eqref{energy equality}
 for any $t\in [0,T].$
\end{Theorem}
\begin{Remark}
The condition \eqref{additional condition on initial data} could be replaced by \eqref{initial velocity assp}.
We will give detail on this issue at the end of this paper.
\end{Remark}

\begin{Remark}
On the incompressible Navier-Stokes equations,  Shinbrot \cite{Shinbrot} has shown that a weak solution can remain its energy if
$$\u\in L^p(0,T;L^q(\O))$$
for any $\frac{1}{p}+\frac{1}{q}\leq \frac{1}{2}$ with $q\geq 4.$ A new proof of Shinbrot's result is given in \cite{Yu}.
Our condition \eqref{condition for velocity} for the compressible version is slightly stronger than the incompressible ones.
Because we have to rely on the commutator lemma (Lemma \ref{Lions's lemma}) to handle the terms $(\rho\u)_t$ and $\Dv(\rho\u\otimes\u)$, which needs additional regularity on $\u$.
%As the same to \cite{Shinbrot}, \eqref{condition for velocity} do not depend on the dimensions.
\end{Remark}

%\begin{Remark} Condition \eqref{bounded away from zero} allows us to have $$\u\in L^{\infty}(0,T;L^2(\O))\cap L^2(0,T;H^1(\O)),$$ in turn which gives us
%\eqref{condition for velocity} in dimensions two. Thus,  the energy conserves only if \eqref{bounded away from zero} and \eqref{additional condition on initial data} true in dimensions two.
%\end{Remark}

\vskip0.3cm

Next we  extend Theorem \ref{main result} to the weak solution of \eqref{CNS} in the following theorem. The global existence of renormalized weak solutions to \eqref{CNS} was established in \cite{L,FNP,F04} for any $\gamma>\frac{N}{2}$. In particular,  we need additional restriction \eqref{bounded away from zero for CNS} for density, which allows us to have the continuity of $\rho(t)$ and $(\sqrt{\rho}\u)(t)$ in the strong topology at $t=0$. Meanwhile, $\u$ is uniformly bounded in $L^2(0,T;H^1(\O))$ for the viscosity constants case. Thus, our approach  allows us to handle vacuum states.
  With the restriction \eqref{bounded away from zero for CNS} in mind, we can modify the proof of Theorem \ref{main result} to show the following result. One slight difference is to show $W=0$ in \eqref{show B is zero} for getting the continuity of $\rho(t)$ and $(\sqrt{\rho}\u)(t)$ in the strong topology at $t=0$. We will give the detail on this point at the end of this paper.
\begin{Theorem}
\label{result for CNS} Let $(\rho,\u)$ be a weak solution of \eqref{CNS}, \eqref{initial data} in the sense of \cite{FNP,Lions}.
Moreover,  if
\begin{equation}
\label{bounded away from zero for CNS}
0\leq \rho(t,x)\leq \bar{\rho}<\infty,\;\;\text{ and }\nabla\sqrt{\rho}\in L^{\infty}(0,T;L^2(\O)),
\end{equation}

\begin{equation}
\label{condition for velocity}
\u \in L^p(0,T;L^q(\O))\quad\text{ for any }  \frac{1}{p}+\frac{1}{q}\leq\frac{5}{12}, \text{ and } q\geq 6 ,
\end{equation}

and
\begin{equation}
\label{initial velocity assp}
\u_0\in L^{k}(\O), \;\frac{1}{k}+\frac{1}{q}\leq \frac{1}{2},
\end{equation}
 then such a weak solution
 $(\rho,\u)$  satisfies
\begin{equation*}
\begin{split}
&\int_{\O}\left(\frac{1}{2}\rho|\u|^2+\frac{\rho^{\gamma}}{\gamma-1}\right)\,dx+2\mu\int_0^T\int_{\O}|\nabla\u|^2\,dx\,dt
\\&\quad\quad\quad\quad\quad\quad+\lambda\int_0^T\int_{\O}|\Dv\u|^2\,dx\,dt=
\int_{\O}\left(\frac{1}{2}\rho_0|\u_0|^2+\frac{\rho_0^{\gamma}}{\gamma-1}\right)\,dx
\end{split}
\end{equation*}
 for any $t\in [0,T].$
\end{Theorem}

\begin{Remark}
An interesting point is that the density may vanish in \eqref{bounded away from zero for CNS}. Thus,
Theorem \ref{result for CNS} indicates that the energy may conserve even on the vacuum.
\end{Remark}

\vskip0.3cm
\bigskip

\section{Proof of main results}

The main object of this section is to prove our main results, including Theorem \ref{main result} and Theorem \ref{result for CNS}.
We devote subsection 2.1 to a proof of Theorem \ref{main result}, and subsection 2.2 to a proof of Theorem \ref{result for CNS}.

% We are able to make use of the same argument to show Theorem \ref{result for CNS}.
\subsection{Proof of Theorem \ref{main result}}
Note that, for any weak solution $(\rho,\u)$, with condition \eqref{bounded away from zero}, it satisfies
\begin{equation}
\label{bounded in H1}
\|\u\|_{L^{\infty}(0,T;L^2(\O))}\leq C<\infty,\quad
\|\nabla\u\|_{L^2(0,T;L^2(\O))}\leq C<\infty.
\end{equation}
Now, let us to give an estimate on $\rho_t$ in the following lemma.
\begin{Lemma}
\label{estimate on density t} For any weak solution $(\rho,\u)$ in the sense of Definition \ref{definition of weak soution} with additional conditions \eqref{bounded away from zero}-\eqref{condition for velocity}, then
$\rho_t$ is bounded in $L^p(0,T;L^{\frac{2q}{q+2}}(\O))+L^{2}(0,T;L^2(\O))$. In particular, $\rho_t$ is bounded in $L^2(0,T;L^{\frac{2q}{q+2}}(\O))$ if p$\geq2.$
\end{Lemma}
\begin{proof}
 Note that
$$\rho_t=-2\sqrt{\rho}\u\cdot\nabla\sqrt{\rho}-\sqrt{\rho}\sqrt{\rho}\Dv\u,$$
consequently $\rho_t$ is bounded in  $L^p(0,T;L^{\frac{2q}{q+2}}(\O))+L^{2}(0,T;L^2(\O))$, thanks to the estimates in definition \ref{definition of weak soution}, \eqref{bounded away from zero} and \eqref{condition for velocity}.
\end{proof}
\begin{Remark}
  For any weak solution $(\rho,\u)$ satisfied B-D entropy \cite{BD,BD2006,BDL}, $\nabla\sqrt{\rho}$ is bounded in $L^{\infty}(0,T;L^2(\O))$. Thus, Lemma \ref{estimate on density t} gives us that $\rho_t$ is locally integrable.
   It is crucial to have the estimate of $\nabla\sqrt{\rho}$ for making commutator estimates.
  % This allows us to make use of Lemma \ref{Lions's lemma} to treat with the terms $(\rho\u)_t$ and $\Dv(\rho\u\otimes\u).$
\end{Remark}
Here we make use of $L^p-L^q$ inequality to deduce the following lemma directly.
\begin{Lemma}
\label{LpLq estimate}
If $\u\in L^{\infty}(0,T;L^2(\O))\cap L^p(0,T;L^q(\O)),$ then there exists some $\alpha\in(0,1)$ such that
$\u\in L^r(0,T;L^s(\O)),$ for any $$\frac{1}{r}=\frac{1-\alpha}{p},$$
and $$\frac{1}{s}=\frac{\alpha}{2}+\frac{1-\alpha}{q}.$$
\end{Lemma}

Now, let us to define $$\overline{f}(t,x)=\eta_{\varepsilon}* f(t,x),\;\;\; t>\varepsilon,$$
where $\eta_{\varepsilon}=\frac{1}{\varepsilon^{N+1}}\eta(\frac{t}{\varepsilon},\frac{x}{\varepsilon})$, and $\eta(t,x)\geq 0$ is a smooth even function compactly supported in the space time ball of radius $1$, and with integral equal to $1$.

The following lemma is crucial in this current paper,  which was proved by Lions in \cite{L}. Here we adopt the following statement of \cite{LV}.
 \begin{Lemma}
 \label{Lions's lemma}
Let $\partial$ be a partial derivative in space or time. Let $f$, $\partial f \in L^p(\R^+\times \O)$, $g\in L^q (\R^+\times \O)$ with $1\leq p,\,q \leq \infty,$ and $\frac{1}{p}+\frac{1}{q}\leq 1. $ Then, we have
$$\|\overline{\partial(fg)}-\partial(f\,\overline{g})\|_{L^r(\R^+\times\O)}\leq C\|\partial f\|_{L^p(\R^+\times\O)}\|g\|_{L^q(\R^+\times\O)}$$
for some constant $C>0$ independent of $\varepsilon$, $f$ and $g$, and with $\frac{1}{r}=\frac{1}{p}+\frac{1}{q}$. In addition,
$$\overline{\partial{(fg)}}-\partial{(f\,\overline{g})}\to 0\quad\text{ in } L^r(\R^+\times\O)$$
as $\varepsilon\to 0$ if $r<\infty.$
 \end{Lemma}

\vskip0.3cm

With Lemma \ref{estimate on density t}-Lemma \ref{Lions's lemma} in hand, we are ready to show Theorem \ref{main result}.
Here we introduce a new function $\Phi=\overline{\psi(t)\overline{\u}}$.
Note that $\psi(t)\in  \mathfrak{D}(0,+\infty)$ is a test function,
where $\mathfrak{D}(0,+\infty)$ is a class of all smooth compactly supported functions in $(0,+\infty)$. In particular,  this function vanishes  close $t=0$.
However, it is needed to extend the result for $\psi(t)\in  \mathfrak{D}(-1,+\infty)$.\\
 Note that, $\psi(t)$ is compactly supported in $(0,\infty).$ $\Phi$ is well defined on $(0,\infty)$ for $\varepsilon$ small enough. Using it to test the second equation in \eqref{NS equation}, one obtains
\begin{equation*}
\int_0^T\int_{\O}\overline{\Phi}\left((\rho\u)_t+\Dv(\rho\u\otimes\u)+\nabla \rho^{\gamma}-\Dv(\rho\mathbb{D}\u)\right)\,dx\,dt=0,
\end{equation*}
which in turn yields
\begin{equation}
\label{weak formulation}
\int_0^T\int_{\O}\psi(t)\overline{\u}\overline{\left((\rho\u)_t+\Dv(\rho\u\otimes\u)+\nabla \rho^{\gamma}-\Dv(\rho\mathbb{D}\u)\right)}\,dx\,dt=0,
\end{equation}
where we used a fact  $\eta(-t,-x)=\eta(t,x).$

The first term in \eqref{weak formulation} shows that
\begin{equation}
\begin{split}
\label{the first term}
\int_0^T\int_{\O}\psi(t)\overline{\u}\overline{(\rho\u)_t}\,dx&=\int_0^T\int_{\O}\psi(t)\left(\overline{(\rho\u)_t}-(\rho\overline{\u})_t\right)\overline{\u}\,dx\,dt+\int_0^T\int_{\O}\psi(t)(\rho\overline{\u})_t\overline{\u}\,dx\,dt
\\&=A +\int_0^T\int_{\O}\psi(t)\rho\partial_t{\frac{|\overline{\u}|^2}{2}}\,dx\,dt+\int_0^T\int_{\O}\psi(t)\rho_t|\overline{\u}|^2\,dx\,dt.
\end{split}
\end{equation}
Similarly, the second term in \eqref{weak formulation} gives us
\begin{equation}
\begin{split}
\label{the second term}
\int_0^T\int_{\O}\psi(t)&\overline{\u}\overline{\Dv(\rho\u\otimes\u)}\,dx\,dt=
\int_0^T\int_{\O}\psi(t)\left(\Dv(\overline{\rho\u\otimes\u)}-\Dv(\rho\u\otimes\overline{\u})\right)\overline{\u}\,dx\,dt
\\&\quad\quad\quad\quad\quad\quad\quad\quad\quad+\int_0^T\int_{\O}\psi(t)\Dv(\rho\u\otimes\overline{\u})\overline{\u}\,dx\,dt
\\&=B+\int_0^T\int_{\O}\psi(t)\rho\u\cdot\nabla\frac{|\overline{\u}|^2}{2}\,dx+\int_0^T\int_{\O}\psi(t)\Dv(\rho\u)|\overline{\u}|^2\,dx\,dt
\\&=B+\int_0^T\int_{\O}\psi(t)\rho_t\frac{|\overline{\u}|^2}{2}\,dx+\int_0^T\int_{\O}\psi(t)\Dv(\rho\u)|\overline{\u}|^2\,dx\,dt.
\end{split}
\end{equation}
Thanks to Lemma \ref{estimate on density t}, the last term of the right-hand side in \eqref{the first term} and the second term of the right-hand side in \eqref{the second term} are well defined. Combining \eqref{the first term}-\eqref{the second term}, the first two terms in \eqref{weak formulation} are given by
\begin{equation}
\label{the first two terms}
\int_0^T\int_{\O}\psi(t)\left(\frac{1}{2}\rho|\overline{\u}|^2\right)_t\,dx\,dt+A+B.
\end{equation}
Next, the last term in \eqref{weak formulation} gives us
\begin{equation}
\label{the last term}
\begin{split}
\int_0^T\int_{\O}\psi(t)\overline{\Dv(\rho\mathbb{D}\u)}\overline{\u}\,dx\,dt&=
\int_0^T\int_{\O}\psi(t)\left(\overline{\Dv(\rho\mathbb{D}\u)}-\Dv(\rho\overline{\mathbb{D}\u})\right)\overline{\u}\,dx\,dt
\\&+\int_0^T\int_{\O}\psi(t)\Dv(\rho\overline{\mathbb{D}\u})\overline{\u}\,dx\,dt
\\&=D-\int_0^T\int_{\O}\psi(t)\rho|\overline{\mathbb{D}\u}|^2\,dx\,dt.
\end{split}
\end{equation}
 We estimate the third term in \eqref{weak formulation} as follows
\begin{equation}
\label{the third term}
\begin{split}
&\int_0^T\int_{\O}\psi(t)\overline{\nabla\rho^{\gamma}}\overline{\u}\,dx\,dt=
\frac{\gamma}{\gamma-1}\int_0^T\int_{\O}\psi(t)\overline{\rho\nabla\rho^{\gamma-1}}\overline{\u}\,dx\,dt
\\&=\frac{\gamma}{\gamma-1}\int_0^T\int_{\O}\psi(t)\left(\overline{\rho\nabla\rho^{\gamma-1}}-\rho\overline{\nabla\rho^{\gamma-1}}\right)\overline{\u}\,dx\,dt
+\frac{\gamma}{\gamma-1}\int_0^T\int_{\O}\psi(t)\rho\overline{\nabla\rho^{\gamma-1}}\overline{\u}\,dx\,dt
\\&=E-\frac{\gamma}{\gamma-1}\int_0^T\int_{\O}\psi(t)\Dv(\rho\overline{\u})\overline{\rho^{\gamma-1}}\,dx\,dt.
\end{split}
\end{equation}
The second term in the last equality of \eqref{the third term} shows that
\begin{equation}
\label{the second part of pressure term}
\begin{split}&
\frac{\gamma}{\gamma-1}\int_0^T\int_{\O}\psi(t)\Dv(\rho\overline{\u})\overline{\rho^{\gamma-1}}\,dx\,dt\\&
=
\frac{\gamma}{\gamma-1}\int_0^T\int_{\O}\psi(t)\left(\Dv(\rho\overline{\u})-\overline{\Dv(\rho\u)}\right)\overline{\rho^{\gamma-1}}\,dx\,dt+\frac{\gamma}{\gamma-1}\int_0^T\int_{\O}\psi(t)\overline{\Dv(\rho\u)} \,\,\overline{\rho^{\gamma-1}}\,dx\,dt
\\&=E_{21}-\frac{\gamma}{\gamma-1}\int_0^T\int_{\O}\psi(t)\overline{\rho}_t\,\,\overline{\rho^{\gamma-1}}\,dx\,dt.
\end{split}
\end{equation}
Meanwhile, the last term in \eqref{the second part of pressure term} is as follows
\begin{equation}
\label{the third part of pressure term}
\begin{split}&
\frac{\gamma}{\gamma-1}\int_0^T\int_{\O}\psi(t)\overline{\rho}_t\,\,\overline{\rho^{\gamma-1}}\,dx\,dt
\\&=
\frac{\gamma}{\gamma-1}\int_0^T\int_{\O}\psi(t)(\overline{\rho}_t-\rho_t)\,\,\overline{\rho^{\gamma-1}}\,dx\,dt+
\frac{\gamma}{\gamma-1}\int_0^T\int_{\O}\psi(t)\,\left(\rho_t\overline{\rho^{\gamma-1}}-\rho_t\rho^{\gamma-1}\right)\,dx\,dt
\\&+
\frac{\gamma}{\gamma-1}\int_0^T\int_{\O}\psi(t)\rho_t\rho^{\gamma-1}\,dx\,dt
\\&=E_{31}+E_{32}+\frac{1}{\gamma-1}\int_0^T\int_{\O}\psi(t)(\rho^{\gamma})_t\,dx\,dt
\end{split}
\end{equation}
Thanks to \eqref{the first two terms}-\eqref{the third part of pressure term}, \eqref{weak formulation} gives us
\begin{equation*}
\begin{split}
&\int_0^T\int_{\O}\psi(t)\left(\frac{1}{2}\rho|\overline{\u}|^2+\frac{\rho^{\gamma}}{\gamma-1}\right)_t\,dx\,dt
\\&+\int_0^T\int_{\O}\psi(t)\rho|\overline{\mathbb{D}\u}|^2\,dx\,dt
+A+B-D+E-E_{21}+E_{31}+E_{32}=0.
\end{split}
\end{equation*}
This yields,
\begin{equation}
\begin{split}
\label{weak energy equality}
&-\int_0^T\int_{\O}\psi_t\left(\frac{1}{2}\rho|\overline{\u}|^2
+\frac{\overline{\rho^{\gamma}}}{\gamma-1}\right)\,dx\,dt
\\&\quad\quad\quad\quad\quad\quad\quad+\int_0^T\int_{\O}\psi(t)\rho|\overline{\mathbb{D}\u}|^2\,dx\,dt
+R_{\varepsilon}(t,x)=0,
\end{split}
\end{equation}
where $R_{\varepsilon}(t,x)=
A+B-D+E-E_{21}+E_{31}+E_{32}.$\\

Note that \eqref{bounded away from zero} and \eqref{condition for velocity}, one obtains
\begin{equation}
\label{convegence first term}
\int_0^T\int_{\O}\frac{1}{2}\rho|\overline{\u}|^2\psi_t\,dx\,dt\to \int_0^T\int_{\O}\frac{1}{2}\rho|\u|^2\psi_t\,dx\,dt\;\;\text{as }\;\varepsilon\to0.
\end{equation}

Thanks to \eqref{bounded away from zero} and \eqref{bounded in H1}, we find that, for any $\varepsilon$ tends to zero,
\begin{equation}
\int_0^T\int_{\O}\psi(t)\rho|\overline{\mathbb{D}\u}|^2\,dx\,dt\to \int_0^T\int_{\O}\psi(t)\rho|\mathbb{D}\u|^2\,dx\,dt.
\end{equation}

\vskip0.3cm

The next step is to make use of Lemma \ref{Lions's lemma} to prove  \begin{equation}
\label{the rest goes to zero}
R_{\varepsilon}(t,x)\to 0
\end{equation}
 as $\varepsilon$ goes to zero.
First, we assume that $\u$ is bounded in $L^r(0,T;L^s(\O))$. We will improve this restriction later. Under this restriction,
 by Lemma \ref{estimate on density t}, $\rho_t$ is uniformly bounded in $L^r(0,T;L^{\frac{2s}{s+2}}(\O))+L^{2}(0,T;L^2(\O))$.
Thus, Lemma \ref{Lions's lemma} gives us
\begin{equation*}
\begin{split}
|A|&\leq\|\psi(t)\|_{L^{\infty}(0,T)}\int_0^T\int_{\O}\left|\overline{\u}[\overline{(\rho\u)_t}-(\rho\overline{\u})_t]\right|\,dx\,dt
\\&\leq C\|\psi(t)\|_{L^{\infty}(0,T)}\|\rho_t\|_{L^2(0,T;L^{\frac{2s}{2+s}}(\O))}\|\u\|^2_{L^r(0,T;L^s(\O))},
\end{split}
\end{equation*}
for any $r\geq 4$ and $ s\geq 6.$\\

Meanwhile, for any $\u\in L^p(0,T;L^q(\O))$,  by Lemma \ref{LpLq estimate}, one obtains that for $0<\alpha < 1$,
\begin{equation}
\label{relationship p q}
 \frac{1-\alpha}{p}= \frac{1}{r},\quad\text{ and } \frac{\alpha}{2}+\frac{1-\alpha}{q}= \frac{1}{s}
\end{equation}with $q\geq 6.$ This gives us, for any $0<\alpha<1$
$$(\frac{1}{p}+\frac{1}{q})(1-\alpha)\leq (\frac{1}{p}+\frac{1}{q})(1-\alpha)+\frac{\alpha}{12}\leq \frac{5}{12}(1-\alpha)$$
and thus  $$\frac{1}{p}+\frac{1}{q}\leq\frac{5}{12}$$
with  $q\geq 6.$

 Thanks to Lemma \ref{Lions's lemma} again, as $\varepsilon$ tends to zero, we have
\begin{equation}
\label{term A goes to 0}
\begin{split}
A\to 0.
\end{split}
\end{equation} \\
 An argument similar to that given above
in the control for $A$ shows that the terms $B,\,D,\,E_{21}$ in $R_{\varepsilon}(t,x)$ here converge  to zero, as $\varepsilon\to 0$, in particular,
\begin{equation}
\label{other rest terms}
B-D-E_{21}\to 0.
\end{equation}
Note that $\rho_t$ is bounded in $L^p(0,T;L^{\frac{2q}{q+2}}(\O))+L^{2}(0,T;L^2(\O))$ and $\rho$ is bounded in $L^{\infty}(0,T;\O)$, we conclude that
$E_{31},\,E_{32}$ goes to zero as $\varepsilon$ tends to zero. Similar argument can show that $E\to 0$ as $\varepsilon$ goes to zero. Thus,
$$R_{\varepsilon}\to 0$$
 as $\varepsilon\to 0.$\\

We are ready to pass to the limits in \eqref{weak energy equality}. Letting $\varepsilon$ goes to zero, using \eqref{convegence first term}-\eqref{the rest goes to zero}, what we have proved is that in the limit,
\begin{equation}
\label{limit equation first}
-\int_0^T\int_{\O}\psi_t\left(\frac{1}{2}\rho|\u|^2+\frac{\rho^{\gamma}}{\gamma-1}\right)\,dx\,dt+\int_0^T\int_{\O}\psi(t)\rho|\mathbb{D}\u|^2\,dx\,dt=
0,
\end{equation}
for any test function $\psi\in \mathfrak{D}(0,\infty).$\\

The final step is to extend our result \eqref{limit equation first} for the test function $\psi(t)\in \mathfrak{D}(-1,\infty).$ To this end, it is necessary for us to have the continuity of $\rho(t)$ and $(\sqrt{\rho}\u)(t)$ in the strong topology at $t=0$. Adopting the similar argument in Vasseur-Yu \cite{VY}, what we expected can be done.\\

Here we remark that $\sqrt{\rho}\u$ is bounded in $L^{\infty}(0,T;L^2(\O)).$
Using  \eqref{bounded away from zero} and
$$\rho_t=-\Dv(\rho\u),$$
 we have \begin{equation}
 \label{density t bound}\rho_t\in L^2(0,T;H^{-1}(\O)).
 \end{equation}

 Note that, $\nabla\sqrt{\rho}$ is bounded in $L^{\infty}(0,T;L^2(\O))$ and $\rho$ is bounded away from infinity. Thus,
the identity $$\nabla\rho=2\sqrt{\rho}\nabla\sqrt{\rho}$$
gives us
\begin{equation}
\label{density H1}
\nabla\rho\in L^{2}(0,T;L^2(\O)).
\end{equation}

From \eqref{density t bound} and \eqref{density H1},
one obtains \begin{equation}
\label{continuity for density}\rho\in C([0,T];L^2(\O)),
\end{equation}
thanks to Theorem 3 on page 287 of book \cite{Evans}.
\\
 Using $$\sqrt{\rho}_t=-\Dv(\sqrt{\rho}\u)+\frac{1}{2}\sqrt{\rho}\Dv\u,$$
 we deduce
 \begin{equation*}\sqrt{\rho}_t\in L^2(0,T;H^{-1}(\O)).
 \end{equation*}
 Because $\nabla\sqrt{\rho}$ is bounded in $L^{\infty}(0,T;L^2(\O))$, % we apply Theorem 3 on page 287, see \cite{Evans} again,
 then
 \begin{equation}
\label{half-continuity for density}\sqrt{\rho}\in C([0,T];L^2(\O)).
\end{equation}

In the meanwhile, we see
\begin{equation}
\begin{split}
\label{control of sqrt density and u}
&\text{ess}\limsup_{t\to0}\int_{\O}|\sqrt{\rho}\u-\sqrt{\rho_0}\u_0|^2\,dx
\\&\leq\text{ess}\limsup_{t\to0}\left(\int_{\O}(\frac{1}{2}\rho|\u|^2+\frac{\rho^{\gamma}}{\gamma-1})\,dx-\int_{\O}(\frac{1}{2}\rho_0|\u_0|^2+\frac{\rho_0^{\gamma}}{\gamma-1})|^2\,dx\right)
\\&+\text{ess}\limsup_{t\to0}\left(2\int_{\O}\sqrt{\rho_0}\u_0(\sqrt{\rho_0}\u_0-\sqrt{\rho}\u)\,dx+\int_{\O}(\frac{\rho_0^{\gamma}}{\gamma-1}-\frac{\rho^{\gamma}}{\gamma-1})\right)
.
\end{split}
\end{equation}

Using \eqref{energy inequality} and the convexity of $\rho\mapsto \rho^{\gamma}$, we have
\begin{equation}
\begin{split}
\label{the rest term for initial time}
&\text{ess}\limsup_{t\to0}\int_{\O}|\sqrt{\rho}\u-\sqrt{\rho_0}\u_0|^2\,dx
\\&\leq 2\text{ess}\limsup_{t\to0}\int_{\O}\sqrt{\rho_0}\u_0(\sqrt{\rho_0}\u_0-\sqrt{\rho}\u)\,dx
\\&=W.
\end{split}
\end{equation}

To show the continuity of $(\sqrt{\rho}\u)(t)$ in the strong topology at $t=0$, we need $W=0$. To this end,
we introduce the time evolution of the integral averages
$$t\in(0,T)\longmapsto\int_{\O}(\rho\u)(t,x)\cdot\psi(x)\,dx,$$ which
is defined by
\begin{equation}
\begin{split}
\label{integral for mass}
&\frac{d}{dt}\int_{\O}(\rho\u)(t,x)\cdot\psi(x)\,dx
\\&\quad\quad\quad\quad=
\int_{\O}\rho\u\otimes\u:\nabla \psi\,dx+\int_{\O}\rho^{\gamma}\Dv\psi\,dx+\int_{\O}\rho\mathbb{D}\u\nabla\psi\,dx,
\end{split}
\end{equation}
where $\psi(x)\in C_0^{\infty}(\O)$ is a test function.
All estimates from \eqref{energy inequality}, \eqref{bounded away from zero} and \eqref{condition for velocity} imply \eqref{integral for mass} is continuous function with respect to $t\in[0,T].$
On the other hand, we have $$\rho\u\in L^{\infty}(0,T;L^{2}(\O)),$$
and hence
\begin{equation}
\label{weak continuous}
\rho\u\in C([0,T]; L^{2}_{\text{weak}}(\O)).
\end{equation}
We consider $W$ as follows
\begin{equation}\begin{split}
\label{show B is zero}
& W=2\text{ess}\limsup_{t\to0}\int_{\O}\frac{\sqrt{\rho_0}\u_0}{\sqrt{\rho}}(\sqrt{\rho\rho_0}\u_0-\rho\u)\,dx
\\&\leq 2\text{ess}\limsup_{t\to0}\int_{\O}\frac{\sqrt{\rho_0}\u_0}{\sqrt{\rho}}(\sqrt{\rho\rho_0}\u_0-\rho_0\u_0)\,dx
+ 2\text{ess}\limsup_{t\to0}\int_{\O}\frac{\sqrt{\rho_0}\u_0}{\sqrt{\rho}}(\rho_0\u_0-\rho\u)\,dx,
\end{split}
\end{equation}
where we used \eqref{bounded away from zero}.

Using \eqref{half-continuity for density}
and \eqref{weak continuous} in \eqref{show B is zero},  one deduces $W=0$ provided that $\sqrt{\rho_0}\u_0\in L^4(\O).$
 Thus, we have
\begin{equation*}
\text{ess}\limsup_{t\to0}\int_{\O}|\sqrt{\rho}\u-\sqrt{\rho_0}\u_0|^2\,dx=0,
\end{equation*}
which gives us
\begin{equation}
\label{continuity in strong topology 2}\sqrt{\rho}\u\in C([0,T];L^2(\O)).
\end{equation}

 By \eqref{continuity for density} and \eqref{continuity in strong topology 2}, we get
 \begin{equation}
 \label{limit000}\lim_{\tau\to0}\frac{1}{\tau}\int_0^{\tau}\int_{\O}\left(\frac{1}{2}\rho|\u|^2+\frac{\rho^{\gamma}}{\gamma-1}\right)\,dx\,dt=\int_{\O}\left(\frac{1}{2}\rho_0|\u_0|^2+\frac{\rho_0^{\gamma}}{\gamma-1}\right)\,dx.
 \end{equation}
 \\

 Choosing the following test function for \eqref{limit equation first},
 \begin{equation*}
 \begin{split}
 &\psi_{\tau}(t)=\psi(t)\;\;\text{ for } \,t\geq \tau+\frac{1}{K},\;\;\psi_{\tau}(t)=\frac{t}{\tau}\;\;\text{ for } t\leq \tau, \;\; \psi_{\tau} \text{ is a $C^1$ smooth function,}
 \end{split}
 \end{equation*}then
 we get
 \begin{equation}
\begin{split}
\label{AASS}
&-\int_{\tau+\frac{1}{K}}^T\int_{\O}\psi_t\left(\frac{1}{2}\rho|\u|^2+\frac{\rho^{\gamma}}{\gamma-1}\right)\,dx\,dt+\int_0^T\int_{\O}\psi_{\tau}(t)\rho|\mathbb{D}\u|^2\,dx\,dt
\\
&-\int_{\tau}^{\tau+\frac{1}{K}}\int_{\O}(\psi_{\tau})_t\left(\frac{1}{2}\rho|\u|^2+\frac{\rho^{\gamma}}{\gamma-1}\right)\,dx\,dt
=\frac{1}{\tau}\int_0^{\tau}\int_{\O}\left(\frac{1}{2}\rho|\u|^2+\frac{\rho^{\gamma}}{\gamma-1}\right)\,dx\,dt.
\end{split}
\end{equation}
Note that, \begin{equation*}
\begin{split}&
\left|\int_{\tau}^{\tau+\frac{1}{K}}\int_{\O}(\psi_{\tau})_t\left(\frac{1}{2}\rho|\u|^2+\frac{\rho^{\gamma}}{\gamma-1}\right)\,dx\,dt\right|
\\&\leq \frac{C}{K}\int_{\O}\left(\frac{1}{2}\rho|\u|^2+\frac{\rho^{\gamma}}{\gamma-1}\right)\,dx
\to0
\end{split}
\end{equation*}
as $K$ goes to large. Letting $K\to\infty,$ from \eqref{AASS} we derive
\begin{equation}
\begin{split}
\label{key limit-aaa}
&-\int_{\tau}^T\int_{\O}\psi_t\left(\frac{1}{2}\rho|\u|^2+\frac{\rho^{\gamma}}{\gamma-1}\right)\,dx\,dt+\int_0^T\int_{\O}\psi_{\tau}(t)\rho|\mathbb{D}\u|^2\,dx\,dt
\\
&
=\frac{1}{\tau}\int_0^{\tau}\int_{\O}\left(\frac{1}{2}\rho|\u|^2+\frac{\rho^{\gamma}}{\gamma-1}\right)\,dx\,dt.
\end{split}
\end{equation}
%By \eqref{half-continuity for density}, \eqref{continuity in strong topology 2} and
By \eqref{limit000},
 passing into the limit as $\tau\to 0$ in \eqref{key limit-aaa}, one obtains
 \begin{equation}
\label{the second level limit}
\begin{split}
-\int_{0}^T\int_{\O}\psi_t\left(\frac{1}{2}\rho|\u|^2+\frac{\rho^{\gamma}}{\gamma-1}\right)\,dx\,dt&+\int_0^T\int_{\O}\psi(t)\rho|\mathbb{D}\u|^2\,dx\,dt
\\&
=\int_{\O}\left(\frac{1}{2}\rho_0|\u_0|^2+\frac{\rho_0^{\gamma}}{\gamma-1}\right)\,dx.
\end{split}
\end{equation}
Taking \begin{equation}\label{psi_name}
\psi(t)\begin{cases}=0 \;\;\;\;\;\quad\quad\quad\quad\text{ if }t\leq \tilde{t}-\frac{\epsilon}{2}
%\\ \text{ is a  continuous function,}\text{ if } n\leq y\leq 2n,
\\ =\frac{1}{2}+\frac{t-\tilde{t}}{\epsilon}\;\;\;\;\;\;\;\;\;\;\text{ if } \tilde{t}-\frac{\epsilon}{2}\leq t\leq \tilde{t}+\frac{\epsilon}{2}
\\ =1\,\;\;\;\;\quad\quad\;\quad\quad\text{ if }  t\geq \tilde{t}+\frac{\epsilon}{2},
\end{cases}\end{equation}
then \eqref{the second level limit} gives for every $\tilde{t}\geq \frac{\epsilon}{2},$
\begin{equation}
\label{the third level limit}
\begin{split}&
\frac{1}{\epsilon}\int_{\tilde{t}-\frac{\epsilon}{2}}^{\tilde{t}+\frac{\epsilon}{2}}\left(\int_{\O}\left(\frac{1}{2}\rho|\u|^2+\frac{\rho^{\gamma}}{\gamma-1}\right)\,dx\right)\,dt
+\int_{\tilde{t}-\frac{\epsilon}{2}}^{\tilde{t}+\frac{\epsilon}{2}}\int_{\O}(\frac{1}{2}+\frac{t-\tilde{t}}{\epsilon})\rho|\mathbb{D}\u|^2\,dx\,dt
\\&+\int_{\tilde{t}+\frac{\epsilon}{2}}^T\int_{\O}\rho|\mathbb{D}\u|^2\,dx\,dt
=\int_{\O}\left(\frac{1}{2}\rho_0|\u_0|^2+\frac{\rho_0^{\gamma}}{\gamma-1}\right)\,dx.
\end{split}
\end{equation}
The second term of  left hand side in \eqref{the third level limit} is controlled as follows
\begin{equation*}
\int_{\tilde{t}-\frac{\epsilon}{2}}^{\tilde{t}+\frac{\epsilon}{2}}\int_{\O}(\frac{1}{2}+\frac{t-\tilde{t}}{\epsilon})\rho|\mathbb{D}\u|^2\,dx\,dt
\leq
\int_{\tilde{t}-\frac{\epsilon}{2}}^{\tilde{t}+\frac{\epsilon}{2}}\int_{\O}\rho|\mathbb{D}\u|^2\,dx\,dt\to 0
\end{equation*}as $\epsilon\to 0.$
Thanks to the Lebesgue point Theorem, \eqref{the third level limit} gives us
\begin{equation*}
\begin{split}&\int_{\O}\left(\frac{1}{2}\rho|\u|^2+\frac{\rho^{\gamma}}{\gamma-1}\right)\,dx
+\int_{0}^T\int_{\O}\rho|\mathbb{D}\u|^2\,dx\,dt
=\int_{\O}\left(\frac{1}{2}\rho_0|\u_0|^2+\frac{\rho_0^{\gamma}}{\gamma-1}\right)\,dx.
\end{split}
\end{equation*}
This ends our proof of Theorem \ref{main result}.\\

\vskip0.3cm

\subsection{Proof of Theorem \ref{result for CNS}}

Now, let us to address the proof of Theorem \ref{result for CNS}.
 We can  modify the above proof slightly to show Theorem \ref{result for CNS}. The main advantage of condition $0<\underline{\rho}\leq \rho(t,x)$ is to have $\nabla\u\in L^2(0,T;L^2(\O))$ in Theorem \ref{main result}. Thus, we can drop this restriction  for the constant viscosities case.
 \\

 The one difference is to show $W=0$ in \eqref{show B is zero}. In fact, we have
\begin{equation*}\begin{split}
& W=2\text{ess}\limsup_{t\to0}\int_{\O}\u_0(\rho_0\u_0-\sqrt{\rho_0\rho}\u)\,dx
\\&\leq 2\text{ess}\limsup_{t\to0}\int_{\O}\u_0(\rho_0\u_0-\rho\u)\,dx
+ 2\text{ess}\limsup_{t\to0}\int_{\O}\u_0(\sqrt{\rho}-\sqrt{\rho_0})\sqrt{\rho}\u\,dx.
\end{split}
\end{equation*}
Applying the similar argument, we can show $$
\rho\u\in C([0,T]; L^{2}_{\text{weak}}(\O)),\;\;\sqrt{\rho}\in C(0,T;L^2(\O)).$$ %Note that, $\nabla\rho$ is uniformly bounded in $L^{\infty}(0,T;L^2(\O)).$
To show $W=0$, we need $\u_0\in L^{k}$ where $\frac{1}{k}+\frac{1}{q}\leq \frac{1}{2}.$ Here, we have to mention this is also true for the degenerate viscosity.\\

After modifying the above one, we are able to follow the same argument to   show Theorem \ref{result for CNS}.

\vskip0.3cm
\bigskip
\section*{Acknowledgments}
The author would like to thank his mentor, Alexis Vasseur, for his  discussions and for his constant encouragement and guidance.


\bigskip\bigskip
\begin{thebibliography}{99}


\bibitem{BD} D. Bresch, B. Desjardins, \emph{Existence of global weak solutions for 2D viscous shallow water equations and convergence to the quasi-geostrophic model.} Comm. Math. Phys., 238 (2003), no.1-3, 211-223.


 \bibitem{BD2006}   D. Bresch and B. Desjardins, \emph{On the construction of approximate solutions for the 2D viscous shallow water model and for compressible Navier-Stokes models.} J. Math. Pures Appl. (9) 86 (2006), no. 4, 362-368.
\bibitem{BDL}
D. Bresch, B. Desjardins, C.-K. Lin,\emph{ On some compressible fluid models: Korteweg, lubrication, and shallow water systems.}
Comm. Partial Differential Equations 28 (2003), no. 3-4, 843-868.


\bibitem{BDZ}D. Bresch, B. Desjardins, E. Zatorska,\emph{Two-velocity hydrodynamics in fluid mechanics: Part II. Existence of global $\kappa$-entropy solutions to the compressible Navier-Stokes systems with degenerate viscosities.} J. Math. Pures Appl. (9) 104 (2015), no. 4, 801-836.


\bibitem{BDIS}T. Buckmaster, C. De Lellis, P. Isett, and L. Sz´ekelyhidi, Jr. \emph{Anomalous dissipation for 1/5-Holder Euler flows.}
Ann. of Math. (2), 182(1):127-172, 2015.
\bibitem{BDS} T. Buckmaster, C. De Lellis, and L. Sz´ekelyhidi, Jr.\emph{ Dissipative Euler flows with Onsager-critical spatial regularity.}
Comm. Pure and Appl. Math., 2015.


\bibitem{CCFS} A. Cheskidov, P. Constantin, S. Friedlander, and R. Shvydkoy.\emph{ Energy conservation and Onsager's conjecture for
the Euler equations.} Nonlinearity, 21(6):1233-1252, 2008.

\bibitem{CET} P. Constantin, W. E and E. Titi, \emph{Onsager's conjecture on the energy conservation for solutions of Euler's equation.}
Comm. Math. Phys. 165 (1994), no. 1, 207-209.


\bibitem{Evans} L. C. Evans, \emph{Partial differential equations}. Second edition, Graduate Studies in Mathematics, 19. American Mathematical Society, Providence, RI, 2010.

\bibitem{Ey} G. L. Eyink. \emph{Energy dissipation without viscosity in ideal hydrodynamics. }I. Fourier analysis and local energy
transfer. Phys. D, 78(3-4):222-240, 1994.
\bibitem{FNP} E. Feireisl, A. Novotn\'{y}, H. Petzeltov\'{a}, \emph{ On the
existence of globally defined weak solutions to the Navier-Stokes
equations.} J. Math. Fluid Mech. \textbf{3} (2001), 358-392.


\bibitem{F04}E. Feireisl, \emph{Dynamics of viscous compressible fluids.}
Oxford Lecture Series in Mathematics and its Applications, 26.
Oxford Science Publications. The Clarendon Press, Oxford University
Press, New York, 2004.


\bibitem{FGSW}
E. Feireisl, P. Gwiazda, A. Swierczewska-Gwiazda, and E. Wiedemann,
\emph{Regularity and energy conservation for the compressible Euler equations.} Arch. Ration. Mech. Anal. 223 (2017), no. 3, 1375-1395.


%\bibitem{H}E. Hopf,\emph{\"{U}ber die Anfangswertaufgabe f\"{u}r die hydrodynamischen Grundgleichungen, }Math. Nachr. 4 (1951), 213-231.

\bibitem{I}P. Isett, \emph{A Proof of Onsager's Conjecture.} Preprint, 2016. 	arXiv:1608.08301

\bibitem{O} L. Onsager, \emph{Statistical Hydrodynamics.} Nuovo Cimento (Supplemento), 6, 279 (1949)
%\bibitem{Le} J. Leray, \emph{ Sur le mouvement d'un liquide visqueux emplissant l'espace.} (French) Acta Math. 63 (1934), no. 1, 193-248.
\bibitem{LS} T.M. Leslie, R. Shvydkoy, \emph{The energy balance relation for weak solutions of the density-dependent Navier-Stokes equations.} Preprint, 2016.arXiv:1602.08527

%\bibitem{Le} J. Leray, \emph{ Sur le mouvement d'un liquide visqueux emplissant l'espace.} (French) Acta Math. 63 (1934), no. 1, 193-248.
\bibitem{LV}
I. Lacroix-Violet, A. Vasseur,\emph{Global weak solutions to the compressible quantum navier-stokes equation and its semi-classical limit}. Preprint, 2016. arXiv:1607.06646

\bibitem{LX} J. Li, Z.-X. Xin, \emph{
Global existence of weak solutions to the barotropic compressible Navier-Stokes flows with degenerate viscosities.} 2015, 	arXiv:1504.06826.
\bibitem{L} P.-L. Lions, \emph{
Mathematical topics in fluid mechanics. Vol. 1.
Incompressible models.} Oxford Lecture Series in Mathematics and its Applications, 3. Oxford Science Publications. The Clarendon Press, Oxford University Press, New York, 1996.



\bibitem{Lions} P.-L. Lions, \emph{ Mathematical topics in fluid mechanics. } Vol. 2. Compressible models. Oxford Lecture Series in Mathematics and its Applications, 10. Oxford Science Publications. The Clarendon Press, Oxford University Press, New York, 1998.

\bibitem{MV} A. Mellet, A. Vasseur, \emph{ On the barotropic compressible Navier-Stokes equations}. Comm. Partial Differential Equations 32 (2007), no. 1-3, 431-452.

\bibitem{Serrin} J.
Serrin,
\emph{The initial value problem for the Navier-Stokes equations.} 1963 Nonlinear Problems (Proc. Sympos., Madison, Wis., 1962) pp. 69-98 Univ. of Wisconsin Press, Madison, Wis.

\bibitem{Shinbrot} M.
Shinbrot,\emph{
The energy equation for the Navier-Stokes system.}
SIAM J. Math. Anal. 5 (1974), 948-954.
\bibitem{VY} A. Vasseur, C. Yu, \emph{Existence of global weak solutions for 3D degenerate compressible Navier-Stokes equations.}  Invent. Math. 206 (2016), no. 3, 935-974.
\bibitem{Yu} C.Yu, \emph{A new proof of the energy conservation for the Navier-Stokes equations}. Preprint, 2016, 	arXiv:1604.05697.
\end{thebibliography}
\end{document}